\documentclass[12pt]{amsart}
\usepackage{amsmath, amssymb, mathptmx, mathtools, a4wide}
\usepackage[colorlinks=true,linkcolor=blue,citecolor=green,pdfpagelabels=false]{hyperref}
\usepackage{paralist}
\usepackage[utf8]{inputenc}
\usepackage[T1]{fontenc}
\numberwithin{equation}{section}

\theoremstyle{plain}
\newtheorem{thm}{Theorem}

\newtheorem*{prop*}{Proposition}
\newtheorem{lem}{Lemma}[section]
\newtheorem{prop}[lem]{Proposition}
\newtheorem{cor}[lem]{Corollary}
\newcommand{\thmref}[1]{Theorem~\ref{#1}}
\newcommand{\lemref}[1]{lemma~\ref{#1}}
\newcommand{\propref}[1]{proposition~\ref{#1}}

\theoremstyle{definition}
\newtheorem{rmk}[lem]{Remark}

\newcommand{\mbb}{\mathbb}
\newcommand{\mbf}{\mathbf}

\newcommand{\q}{\quad}
\newcommand{\qq}{\qquad}

\newcommand{\mc}{\mathcal}

\newcommand{\mrm}{\mathrm}

\begin{document}
\title[Hecke eigenvalues of non Saito--Kurokawa lifts of degree 2]{Large Hecke eigenvalues and an Omega result for non Saito--Kurokawa lifts}
\author{Pramath Anamby}
\address{Department of Mathematics\\
Harish-Chandra Research Institute\\
Prayagraj (Allahabad)- 211019, India.}
\email{pramathav@hri.res.in, pramath.anamby@gmail.com}

\author{Soumya Das}
\address{Department of Mathematics\\ 
Indian Institute of Science\\ 
Bangalore -- 560012, India\\
and Humboldt Fellow, Universit\"{a}t Mannheim.}
\email{soumya@iisc.ac.in, sdas@mail.uni-mannheim.de} 

\author{Ritwik Pal}
\address{Department of Mathematics\\
Indian Institute of Science\\
Bangalore -- 560012, India.}
\email{ritwikpal@iisc.ac.in, ritwik.1729@gmail.com}

\begin{abstract}
We prove a result on the distribution of Hecke eigenvalues, $\mu_F(p^r)$ (for $r=1,2$ or $3$) of a non Saito--Kurokawa lift $F$ of degree $2$. As a consequence, we obtain an Omega result for the Hecke eigenvalues for such an $F$, which is the best possible in terms of orders of magnitude.
\end{abstract}
\maketitle

\section{Introduction}
The study of behavior of Hecke eigenvalues has been an interesting as well as an important theme in the theory of modular forms. For example, the distribution of Hecke eigenvalues and Omega results (i.e., `sharp' lower bounds on suitable subsequences) have been studied extensively. In  the case of an elliptic Hecke eigenform, the equidistribution of the eigenvalues is a consequence of Sato--Tate conjecture, which is known from the deep results in \cite{barnet}. However, reasonable Omega results can, in many cases, be proved by less sophisticated techniques. For example it is well known that for holomorphic cusp forms on $\mrm{GL}(2)$, such a result follows from the fact that for $r\le 4$, the symmetric $r$-power $L$-functions have an analytical continuation upto $\mathrm{Re}(s)\ge 1/2$ (see \cite{murty1983oscillations}).

In the case of our interest, namely holomorphic Siegel modular forms of degree $2$, none of the above-mentioned results are known outside of the Maa{\ss} space, even though there are some average results \cite{kim} (vertical Sato--Tate on average) and \cite{tsuzuki} (Sato--Tate on average). There are far fewer results however, when one fixes the modular form. Namely, in the case at hand, the distribution of eigenvalues $\lambda_F(p)$, ($p$ prime, $F$ is a non Saito--Kurokawa lift) can be found in \cite{saha2011prime} and \cite{das2013natural}. 

In this article we study the distribution of Hecke eigenvalues $\mu_F(p^r)$ (for $r=1,2,3$, $p$ being a prime) of a Siegel Hecke eigenform of degree $2$ with full level that is not a Saito--Kurokawa (Maa{\ss}) lift. We do this with an aim of proving an Omega result for Hecke eigenvalues of such an $F$. Let us denote by $S_k^{2,*}$, the Maa{\ss} subspace and by $S_k^{2,\perp}$ subspace of $S_k^{2}$ orthogonal to $S_k^{2,*}$.  Our main result \thmref{th:dis} implies (via \thmref{th:omega}) in particular that the Ramanujan--Petersson conjecture for eigenforms in $S_k^{2,\perp}$ is optimal, in a sense described below.

First, let us describe the main result of this article. Let $F \in S_k^{2,\perp}$ be a Hecke eigenform with eigenvalues $\mu_F(n)$; so that if $T(n)$ denotes the $n$-th (similitude) Hecke operator on $S^2_k$, one has $T(n) F = \mu_F(n) F$ for all $n \ge 1$. The Ramanujan--Petersson conjecture (proved by Weissauer, \cite{weissauer2009endoscopy}) for $F$ implies that (see \cite{das-koh-sen}) for all $n \ge 1$,
\begin{align}\label{rama}
|\mu_F(n) | \le d_5(n) n^{k-3/2} ,
\end{align}
where $d_5(n)$ denotes the number of ways of writing $n$ as the product of $5$ positive integers. We now normalize $\mu_F(n)$ by putting
\begin{align} \label{lnorm}
\lambda_F(n) =\mu_F(n) /n^{k-3/2}.
\end{align}
We call $\mu_F(n) $ to be `large' if $|\lambda_F(n)| >c$ for some $c>1$. Our main theorem then says that there exist a plethora of `large' eigenvalues if we search in the sequence $\{ p^j |  p \text{ prime }, j=1,2,3\}$. By multiplicativity of the Hecke eigenvalues, this would then produce other `large' eigenvalues.

\begin{thm}\label{th:dis}
Let $F\in S_k^{2,\perp}$ be a Hecke eigenform.  Then there exists $c>1$ and $\delta>0$ such that
\begin{equation}
\liminf_{x \to \infty}\frac{\#\{p\le x : \max\{|\lambda_F(p^i)|: i=1,2,3\}\ge c\}}{\pi(x)}>\delta.
\end{equation}
where $\pi(x)$ denotes the number of primes upto $x$.
\end{thm}
Note that this would mean that for every large $x$, there exists an $i=i_x\in\{1,2,3\}$ such that $\#\{p\le x: |\lambda(p^i)|>c\}>\delta\cdot \pi(x)$. This immediately gives us the following corollary.
\begin{cor}\label{cor:dis}
For at least one $j \in \{1,2,3\}$, the following statement is true: there exist constants $c>1$ and $\delta>0$ such that
\begin{align}
\limsup_{x \to \infty} \frac{\# \{p\le x : |\lambda(p^j)|\ge c\}}{\pi(x)} >\delta.
\end{align} 
\end{cor}
The main point to note here is that $c >1$ (so that we are dealing with `large' eigenvalues); analogous assertions when $c<1$ follow already from \cite{das2013natural}.

The main tools used in the proof of Theorem \ref{th:dis} are the prime number theorems for the spinor and the standard $L$-functions (denoted as $Z(F,s)$ and $Z^{st}(F,s)$ respectively) attached to $F$ and the Hecke relations among the Hecke eigenvalues. Also crucially used in the proof is the existence of a functorial transfer from $\mrm{GSp}(4)$ to $\mrm{GL}(4)$ from the work of \cite{pitaleetal}, which enables us to use the analytic machinery from $\mrm{GL}(4)$ automorphic representations.

Let us now discuss some applications of \thmref{th:dis} towards Omega results on the sequence of eigenvalues $\{ \lambda_F(n) \}$. In particular, we would like to know if \eqref{rama} is the best possible. This means two things: first, the exponent $k-3/2$ should be the best possible and second, the order of magnitude of the slowly growing function $d_5(n)$ should also be the best possible. Of these, the assertion about the exponent is true and follows from \cite{das2013natural}. It should also follow by considering the Rankin--Selberg convolution of $Z(F,s)$ with itself, and arguing with the location of poles (cf. \cite[remark~5.3]{das-boech}). More subtle is the slowly growing function, and we prefer to treat these functions simultaneously.

To set the stage, let us recall some facts about this type of questions in the case of elliptic modular forms and Saito--Kurokawa lifts. In the case of elliptic modular forms, the answer to the question of sharpness of the Ramanujan--Petersson conjecture
\[
\text{(Deligne's bound):  } \q |a_g(m)| \le d_2(n) n^{(k-1)/2}, \q n\ge 1
\] 
(for a newform $g$) is that it is the best possible in terms of the exponent $(k-1)/2$; so the question boils down to understanding the behavior of the function $a_g(m) / m^{(k-1)/2}$. One knows the  following $\Omega$-type results about this function. In \cite{rankin1973theta} Rankin proved, essentially exploiting the prime number theorem for a Hecke eigenform $g \in S_k$, that it is not bounded:
\begin{equation} \label{om1}
\limsup_{m \to \infty} \frac{a_g(m)}{m^{(k-1)/2}} = \infty.
\end{equation}
Even a stronger result is known due to Ram Murty (cf. \cite{murty1983oscillations}, using the holomorphy of suitable symmetric power $L$-functions):
\begin{equation} \label{om2}
\frac{a_g(m)}{m^{(k-1)/2}} = \Omega \left(  \exp \left(\frac{\alpha \log m}{  \log \log m }\right)  \right) \qq (\alpha>0).
\end{equation}
It is known that the Saito--Kurokawa lifts of degree 2, fail to satisfy (\ref{rama}). Instead they satisfy
\begin{equation}
\lambda_F(n)\ll_\epsilon n^{k-1+\epsilon}, \q \text{for any } \epsilon>0. 
\end{equation} 
An Omega result for such an $F\in S_k^{2}$ was obtained by Das (see \cite{das2014omega}) and was later improved by Gun et al (see \cite{gun2018hecke}).

Here and in the rest of the paper, for arithmetical functions $f(n),g(n)$ with $g(n)>0$ for all $n \ge 1$, we use the notation
\begin{equation} \label{omegadef}
f(n)= \Omega_{\pm} (g(n))  \q \text{ if and only if } \q \limsup_{n \to \infty} \frac{ f(n) }{g(n)} >0 \q (\text{resp. } \liminf_{n \to \infty}  \frac{ f(n) }{g(n)} <0).
\end{equation}
In more simple terms, this just means that $|f(n) |/g(n)$ is bounded away from zero along a subsequence of the set of natural numbers $\mbb N$. Moreover we write $f(n)= \Omega (g(n))$ if $|f(n)|= \Omega_{+} (g(n))$.

Using the corollary \ref{cor:dis} of \thmref{th:dis}, we can deduce easily the following Omega result. 

\begin{thm}\label{th:omega}
Let $F$ be as in \thmref{th:dis}. Then there exists a constant $c>0$ such that
\begin{equation}
\lambda_F(n)=\Omega_{\pm}\left(\exp\left(\frac{c\log n}{\log\log n}\right)\right).
\end{equation}
\end{thm}
Actually the above Omega result is realized over a certain subset of fourth power-free integers. It is easy to check that on this subset $\log d_5(n)$ and the function $\log n/\log\log n$ are same asymptotically upto the constant $c$. For example, if $p$ is a prime, then $d_5(p) = 5$. Therefore putting $a_N = \prod_{p \le N} p$, $\log d_5(a_N) = \log 5\cdot \log(\pi(N)) \sim A\cdot \log a_N/\log \log a_N $. So our result can also be presented as $\lambda_F(n) = \Omega_{\pm}( d_5(n)^\omega)$ for some $\omega>0$. At any rate this not only proves the optimality of the exponent in \eqref{rama}, but that the slowly growing function is also the same upto a suitable exponent.

It is also interesting to ask for Omega results in the context of Fourier coefficients; this has recently been addressed in \cite{das-omega}. It is not immediately clear how the results of this article influence those of \cite{das-omega} and vice-versa.

\subsection*{Acknowledgments}
{\small S.D. was supported by a Humboldt Fellowship from the  Alexander von Humboldt Foundation at Universit\"{a}t Mannheim during the preparation of the paper, and thanks both for the generous support and for providing excellent working conditions. He also thanks IISc, Bangalore, DST (India) and UGC centre for advanced studies for financial support. During the preparation of this work S.D. was supported by a MATRICS grant MTR/2017/000496 from DST-SERB, India. 

P.A. and R.P. were supported by IISc Research Associateship during the preparation of this article and thank IISc, Bangalore for the support.}

\section{Notation and Preliminaries}
First we recall some basic facts about Siegel cusp forms of degree $2$ and the classical $L$-functions attached to them. Let $F\in S_k^2$ be an eigenform for all Hecke operators $T(n)$ which is not a Saito--Kurokawa lift. Let $\{\lambda_F(n)\}$ (normalized as in \eqref{lnorm}) be the normalized eigenvalues of $F$. We refer the reader to \cite{andrianov1974euler} for more details.

\subsection*{Some L-functions attached to $F$} The degree $4$ spinor zeta function attached to $F$ is given by 
\begin{equation*}
Z(F,s)=\prod_{p}Z_p(F,p^{-s}),
\end{equation*}
where the $p$-th Euler factor $Z_p(F,\cdot)$ of $Z(F, \cdot)$ is given by
\begin{align}\label{spinorp}
Z_p(F,t)^{-1}&=(1-\alpha_{0,p} t)(1-\alpha_{0,p}\alpha_{1,p} t)(1-\alpha_{0,p}\alpha_{2,p} t)(1-\alpha_{0,p}\alpha_{1,p}\alpha_{2,p} t)\\
&=1-\lambda_F(p) t+(\lambda_F(p)^2-\lambda_F(p^2)-p^{-1})t^2-\lambda_F(p) t^3+t^4.\nonumber
\end{align}
The degree $5$ standard $L$-function attached to $F$ is given by
\begin{equation*}
Z^{st}(F,s)=\prod_{p}Z_p^{st}(F,p^{-s}),
\end{equation*}
where
\begin{equation}\label{standardp}
Z_p^{st}(F,t)^{-1}=(1-t)(1-\alpha_{1,p}t)(1-\alpha_{2,p}t)(1-\alpha_{1,p}^{-1}t)(1-\alpha_{2,p}^{-1}t).
\end{equation}
Here $\alpha_{0,p},\alpha_{1,p},\alpha_{2,p}$ denote the Satake $p$-parameters attached to $F$ and satisfy
\begin{equation}\label{alphaprod}
\alpha_{0,p}^2\alpha_{1,p}\alpha_{2,p}=1.
\end{equation}
By virtue of the Ramanujan--Petersson conjecture proved by Weissauer (see \cite{weissauer2009endoscopy}) we have
\begin{equation*}
|\alpha_{0,p}|=|\alpha_{1,p}|=|\alpha_{2,p}|=1,
\end{equation*}
for all primes $p$. Moreover the Hecke eigenvalues are related to the spinor zeta function by
\begin{equation} \label{lamspin}
\sum_{n \ge 1} \lambda_F(n) n^{-s} = \frac{Z(F,s)}{\zeta(2s+1)}.
\end{equation}

Let the Dirichlet series of $Z^{st}(F,s)$ be denoted as
\begin{equation*}
Z^{st}(F,s)=\sum_{n\ge 1}\frac{b(n)}{n^s}.
\end{equation*}
Then by expanding (\ref{standardp}) we have
\begin{equation}\label{bpexp}
b(p)=1+\alpha_{1,p}+\alpha_{2,p}+\alpha_{1,p}^{-1}+\alpha_{2,p}^{-1}.
\end{equation}

From \cite{pitaleetal} we know that there exist cuspidal automorphic representations $\Pi_4$ of $\mrm{ GL}_4(\mathbb{A})$ and $\Pi_5$ of $\mrm{ GL}_5(\mathbb{A})$ such that
\begin{equation*}
Z(F,s)=L(\Pi_4,s)\qq\text{and}\qq Z^{st}(F,s)=L(\Pi_5,s).
\end{equation*}
Then using the prime number theorem (PNT) for Rankin--Selberg L functions  $L(\Pi_4\times\Pi_4,s)$ and $L(\Pi_5\times\Pi_5,s)$ (see \cite{liu-ye}), we obtain
\begin{lem}\label{lem:PNT}
For all large $X$
\begin{enumerate}
\item $\underset{p\le X}{\sum}\lambda_F(p)^2\log p= X+O(X\exp(-\kappa_1\sqrt{\log X})).$
\item $\underset{p\le X}{\sum}b(p)^2\log p= X+O(X\exp(-\kappa_2\sqrt{\log X})).$
\end{enumerate}
Here $\kappa_1, \kappa_2>0$.
\end{lem}
\subsection*{Hecke relations} The eigenvalues $\lambda_F(p^n)$ of $F$ satisfy the following recursive relation (from \cite[Theorem 1.3.2]{andrianov1974euler}).
\begin{equation}\label{heckerelations}
\lambda_F(p^n)=\lambda_F(p)\Big(\lambda_F(p^{n-1})+\lambda_F(p^{n-3})\Big)-\lambda_F(p^{n-2})\Big(\lambda_F(p)^2-\lambda_F(p^2)-\frac{1}{p}\Big)-\lambda_F(p^{n-4}).
\end{equation}
We also need the relation between the eigenvalues $\lambda_F(p)$, $\lambda_F(p^2)$ and the Dirichlet coefficients $b(p)$ of the standard $L$-function $Z^{st}(F,s)$. Let $\beta_{1,p}=\alpha_{0,p}$, $\beta_{2,p}=\alpha_{0,p}\alpha_{1,p}$, $\beta_{3,p}=\alpha_{0,p}\alpha_{2,p}$ and $\beta_{4,p}=\alpha_{0,p}\alpha_{1,p}\alpha_{2,p}$. The $p$-th Euler factor of $Z(F,s)$ can be written in terms of $\beta_{i,p}$s as follows.
\begin{equation*}
Z_p(F,t)^{-1}=\prod_{1\le i\le 4}(1-\beta_{i,p} t).
\end{equation*}
By expanding the product and using (\ref{spinorp}) we get the following identities. Namely
\begin{equation} \label{lambeta}
\lambda_F(p)=\sum_{1\le i\le 4}\beta_{i,p}
\end{equation}
and 
\begin{equation}\label{lambpp2}
\lambda_F(p)^2-\lambda_F(p^2)-p^{-1}=\sum_{1\le i<j\le 4}\beta_{i,p}\beta_{j,p}.
\end{equation}
Combining these identities we get
\begin{equation}\label{lamp^2}
\lambda(p^2)=\sum_{1\le i\le j\le 4}\beta_{i,p}\beta_{j,p}-\frac{1}{p}.
\end{equation}
From \eqref{lambeta}, \eqref{lamp^2} and \eqref{bpexp} one obtains the following estimates.
\begin{equation}\label{lamest}
|\lambda_F(p)| \le 4, \q |\lambda_F(p^2)| \le 10+1/p, \q |b(p)| \le 5.
\end{equation}

We also need the relation between the eigenvalues $\lambda(p), \lambda(p^2)$ and the Dirichlet coefficient $b(p)$ of the standard L-function. We get the following relation by using the identities  (\ref{alphaprod}), (\ref{bpexp}), \eqref{lambeta} and  \eqref{lambpp2}.
\begin{equation}\label{bprelation}
\lambda_F(p)^2-\lambda_F(p^2)=b(p)+1+\frac{1}{p}.
\end{equation}

For $a<b$ and $i=1,2$ or $3$, we consider the following subsets of $\mc P$, the set of prime numbers.
\begin{equation}
V_i(a,b;x):=\{p\le x : a\le |\lambda(p^i)|<b\}
\end{equation}
and we denote the set  $\{p\le x :|\lambda(p^i)|\ge a\}$ by $V_i(a,\bullet;x)$. Let us put
\[ \eta_1=10^{-10} \text{  and  }  \eta_2=1/10. \] 

\section{Proof of Theorem \ref{th:dis}}
In this section we collect various implications arising from the asymptotic formulas of the PNT for $Z(F,s)$ and $Z^{st}(F,s)$ (cf. \lemref{lem:PNT}) in combination with the Hecke relations \eqref{bprelation} and the bounds on the eigenvalues \eqref{lamest}. The results are in the form of lower bounds on the sets $V_j(a,b;x)$ under suitable hypotheses.

Note that using the partial summation one can deduce (from lemma \ref{lem:PNT}) that
\begin{equation}\label{PNT1}
\sum_{p \le x} \lambda_F(p)^2  = \frac{x}{\log x} + o\left( \frac{x}{\log x}\right)
\end{equation}
and similarly
\begin{equation}\label{pntstd}
\sum_{p\le x}b^2(p)=\frac{x}{\log x}+o\left(\frac{x}{\log x}\right).
\end{equation}
\subsection{Choice of $X_0$:}\label{subX0}
We choose a large $X_0$ such that the following hold (note here that $X_0$ may be dependent on the weight $k$).
\begin{enumerate}
\item  Let $M(x)$ and $E_i(x)$ ($i=1,2$) denote the main and error terms in (\ref{PNT1}) and (\ref{pntstd}) respectively. Then, for $x> X_0$, $E_i(x)\le 10^{-6}\cdot M(x)$ for $i=1,2$.
\item  For $x\ge X_0$, $\pi(10^4)\le 10^{-6}\cdot \pi(x)$.
\item $\frac{999}{1000}\cdot\pi(x)\le\frac{x}{\log x}$ for all $x> X_0$.
\end{enumerate}

With this choice of $X_0$, we have the following results.
\begin{prop}\label{prop:proportion1}
For any $x\ge X_0$, one of the following is true.

(i) \, For some $\delta_1\ge 10^{-5}$, \, $|V_1(1+\eta_1, \bullet ;x)|  \ge \delta_1\cdot \pi(x)$.

(ii) \, For some $\delta_2\ge 98/100$,\,
$|V_1(1-\eta_2,1+\eta_1;x)| \ge \delta_2\cdot\pi(x)$.
\end{prop}

\begin{proof}
Let $x_0\ge X_0$ such that $(i)$ and $(ii)$ does not hold. That is suppose $|V_1(1+\eta_1, \bullet ;x_0)|< 10^{-5}\cdot \pi(x_0)$ and $|V_1(1-\eta_2,1+\eta_1;x_0)|<98/100 \cdot\pi(x_0)$.

Now we decompose the sum on the LHS of \eqref{PNT1} into disjoint parts and bound them as follows:
\begin{align*}
\sum_{p\le x_0}\lambda_F(p)^2&=\sum_{p\in V_1(0,1-\eta_2;x_0) }\lambda_F(p)^2+\sum_{p\in V_1(1-\eta_2,1+\eta_1;x_0)}\lambda_F(p)^2+\sum_{p\in V_1(1+\eta_1, \bullet;x_0)}\lambda_F(p)^2\\
&< (1-\eta_2)^2 |V_1(0,1-\eta_2;x_0)| + (1+\eta_1)^2 |V_1(1-\eta_2,1+\eta_1;x_0)|\\
&\q+16 |V_1(1+\eta_1, \bullet;x_0)|.
\end{align*}
For simplicity, let us put 
\[ A:= |V_1(0,1-\eta_2;x_0)|, B:=|V_1(1-\eta_2,1+\eta_1;x_0)|, C:=|V_1(1+\eta_1, \bullet;x_0)| ,\]
so that $A+B+C = \pi(x_0)$. Then 
\begin{align}
\sum_{p\le x_0}\lambda_F(p)^2 &<   (1-\eta_2)^2 \pi (x_0) + B((1+\eta_1)^2 - (1-\eta_2)^2 ) +C(16  -  (1-\eta_2)^2) \nonumber \\
&< \left(  (1-\eta_2)^2 + \frac{98}{100} \left(  (1+\eta_1)^2 - (1-\eta_2)^2 \right) + 10^{-5} \left( 16  -  (1-\eta_2)^2 \right) \right) \pi(x_0)\nonumber \\
&< \frac{998}{1000} \cdot \pi(x_0), \label{ineqcontra}
\end{align}
upon a short calculation. Thus, for any $x$ such that the conditions $(i)$ and $(ii)$ both fail, the RHS is bounded by $998/1000\cdot \pi(x)$. This is clearly a contradiction in view of conditions (1) and (3) in subsection (\ref{subX0}) on choice of $X_0$.
\end{proof}

If condition $(i)$ of \propref{prop:proportion1} is true  for all $x\ge X_0$, then the proof of Theorem \ref{th:dis} is done. But, if  for some $x_0\ge X_0$ only condition $(ii)$ of \propref{prop:proportion1} is true, then we need to look at the sets $V_2(a,b;x_0)$ and $V_3(a,b;x_0)$. To do this we look at the distribution of coefficients $b(p)$ of the standard $L$-function $Z^{st}(F,s)$.

\begin{prop}\label{prop:v_2}
Let $x_0\ge X_0$ be such that the condition (i) of proposition \ref{prop:proportion1} does not hold. Additionally suppose that $\#\{p\le x_0 : |b(p)|>2.1\}>10^{-3}\cdot\pi(x_0)$. Then \[|V_2(1.09,\bullet;x_0)| > 9\times 10^{-4} \cdot \pi(x_0).\]
\end{prop}

\begin{proof}
From (\ref{bprelation}) we have
\begin{equation} \label{tricky}
|\lambda_F(p^2)|\ge |b(p)|-|1-\lambda_F(p)^2+p^{-1}|.
\end{equation}
Now for $p\not\in V_1(1+\eta_1, \bullet;x)$, note that 
\begin{align}
|1-\lambda_F(p)^2+1/p | \le
\begin{cases}
 1+1/p & \q\text{if }\q\q  |\lambda_F(p)| \le 1;\\
\alpha + 1/p & \q\text{if }\;1< |\lambda_F(p)| \le 1+\eta_1.
\end{cases}
\end{align}
In the second inequality above we have put 
$\lambda_F(p)^2 = 1+\alpha$ and an easy calculation shows that $0<\alpha < 10^{-9}$.  Thus it follows from \eqref{tricky} that
\begin{equation}
|\lambda_F(p^2)| > |b(p)|-1-1/p.
\end{equation}

Let us put $A(x):= \{p\le x : |b(p)|>2.1\}$. Moreover, if $p\in A(x)$ (and $p>10^4$)  we have
\begin{equation}
|\lambda_F(p^2)| > 2.1-1-p^{-1}>1.09.
\end{equation}
These observations suffice to finish the proof as follows. From our two hypotheses in the statement of \propref{prop:v_2} it follows that
\begin{equation} \label{av1}
|A(x_0)| > 10^{-3} \pi(x_0) ; \q  |V_1(1+\eta_1, \bullet;x_0)| < 10^{-5} \pi(x_0).
\end{equation}
From the above calculations and \eqref{av1} we then conclude (putting $B^c = \text{ `complement' of } B$)
\begin{equation}
A(x_0) \cap V_1(1+\eta_1, \bullet;x_0)^{c}\setminus\mc P(10^4) \subset V_2(1.09,\bullet;x_0),
\end{equation}
where $\mc P(10^4)$ is the set of primes $\le 10^4$. By our choice of $X_0$, we have $\pi(10^4)\le \pi(x_0)/10^6$. Therefore
\begin{equation}
|V_2(1.09,\bullet;x_0)| \ge |A(x_0)| - |V_1(1+\eta_1, \bullet;x_0)|-\pi(10^4) \ge (10^{-3} - 10^{-5}-10^{-6} ) \pi(x_0),
\end{equation}
which immediately gives the lemma.
\end{proof}

Now we prove a result regarding the coefficients $b(p)$ of $Z^{st}(F,s)$.

\begin{prop}\label{prop:bp1/18}
Let $x_0\ge X_0$ be such that $\#\{p\le x_0 : |b(p)|>2.1\}\le 10^{-3}\cdot\pi(x_0)$. Then\qq
$\#\{p\le x_0 : 6/7\le |b(p)|\le 2.1\}>\frac{1}{16}\cdot\pi(x_0)$.
\end{prop}

\begin{proof}
We argue in the same way as in proposition \ref{prop:proportion1}. First we decompose the LHS of (\ref{pntstd}) into disjoint sums as follows:
\begin{equation}
\sum_{p\le x_0}b^2(p)=\underset{0\le |b(p)|<6/7}{\sum_{p\le x_0}}b^2(p)+\underset{6/7\le |b(p)|\le 2.1}{\sum_{p\le x_0}}b^2(p)+\underset{ 2.1<|b(p)|\le 5}{\sum_{p\le x_0}}b^2(p).
\end{equation}
As in the proof of proposition \ref{prop:proportion1}, let $A$, $B$ and $C$ denote the cardinality of the sets in the first, second and third terms of the RHS, respectively. Thus we have $A+B+C=\pi(x_0)$ and we get
\begin{align*}
\sum_{p\le x_0}b^2(p)&\le \frac{36}{49}(\pi(x_0)-B-C)+4.41\cdot B+25\cdot C\\
&=\frac{36}{49}\cdot \pi(x_0)+(4.41-\frac{36}{49})\cdot B+(25-\frac{36}{49})\cdot C.
\end{align*}
From our assumption we have, $C\le 10^{-3}\cdot \pi(x_0)$. Now, if the conclusion of the proposition is not true, then $B\le \frac{1}{16}\cdot \pi(x_0)$ and  we have
\begin{align*}
\sum_{p\le x_0}b^2(p)&\le \left(\frac{36}{49}+3.68\cdot\frac{1}{16}+24.27\cdot\frac{1}{10^3}\right)\cdot \pi(x_0)\\
&<\frac{989}{1000}\cdot\pi(x_0).
\end{align*}
A clear contradiction to (\ref{pntstd}) by our choice of $X_0$.
\end{proof}

\begin{prop}\label{prop:v_3}
Let $x_0\ge X_0$ be such that the condition (1) of proposition \ref{prop:proportion1} does not hold. Additionally suppose that $\#\{p\le x_0 : |b(p)|>2.1\}\le 10^{-3}\cdot\pi(x_0)$. Then \[| V_3(1.02, \bullet;x_0)|>\frac{1}{25}\cdot \pi(x_0).\]
\end{prop}
\begin{proof}
We again make use of the following inequality from (\ref{bprelation}).
\begin{equation*}
|\lambda_F(p^2)|\ge |b(p)|-|1-\lambda_F(p)^2+p^{-1}|.
\end{equation*}
For $p\in V_1(1-\eta_2,1+\eta_1;x)$, since $\eta_2=\frac{1}{10}$, we have
\begin{align}
|1-\lambda_F(p)^2+1/p | \le
\begin{cases}
 19/100+1/p & \q\text{if }\;(1-\eta_2)\le |\lambda_F(p)| \le 1;\\
\alpha +1/p & \q\text{if }\q \qq 1< |\lambda_F(p)| < 1+\eta_1,
\end{cases}
\end{align}
where $0<\alpha<10^{-9}$. Thus for $p\in V_1(1-\eta_2,1+\eta_1;x)\cap \{p\le x : |b(p)|\ge 6/7\}$ we have
\begin{equation*}
|\lambda_F(p^2)|\ge \frac{6}{7}-\frac{19}{100}-\frac{1}{p}.
\end{equation*}
Again choosing $p$ large enough ($p>10^4$) we get that $|\lambda_F(p^2)|\ge 0.667>2/3$ and we have
\begin{equation}\label{p22/3set}
V_1(1-\eta_2,1+\eta_1;x)\cap \{p\le x : |b(p)|\ge 6/7\} \setminus\mc P(10^4)\subseteq V_2(2/3,\bullet;x).
\end{equation}

Now from the Hecke relations (see (\ref{heckerelations})) we have
\begin{equation}
\lambda_F(p^3)=\lambda_F(p)\left(2\lambda_F(p^2)-\lambda_F(p)^2+1+\frac{1}{p}\right)
\end{equation}
and if $p\in V_2(2/3,\bullet;x_0)\cap V_1(1-\eta_2,1+\eta_1;x_0)$, we have
\begin{align*}
|\lambda_F(p^3)|&\ge (1-\eta_2)\left||2\lambda_F(p^2)|-|\lambda_F(p)^2-1-\frac{1}{p}|\right|\\
&> \frac{9}{10}\left( \frac{4}{3}-\frac{19}{100}-\frac{1}{p}\right)\\
&> 1.02.
\end{align*}
Combining this with (\ref{p22/3set}), gives us the following inclusions. 
\begin{align}
V_3(1.02, \bullet;x_0)&\supseteq V_2(2/3,\bullet;x_0)\cap V_1(1-\eta_2,1+\eta_1;x_0)\nonumber\\
&\supseteq V_1(1-\eta_2,1+\eta_1;x_0)\cap \{p\le x_0 : |b(p)|\ge 6/7\}\setminus\mc P(10^4).
\end{align}
Now since condition (1) of proposition \ref{prop:proportion1} does not hold for $x_0$, $V_1(1-\eta_2,1+\eta_1;x_0)\ge\frac{98}{100}\cdot \pi(x_0)$ and from proposition \ref{prop:bp1/18}, we have $\#\{p\le x_0 : |b(p)|\ge 6/7\}>\frac{1}{16}\cdot\pi(x_0)$. Thus
\begin{equation}\label{p22/3}
|V_3(1.02, \bullet;x_0)|\ge \left(\frac{98}{100}+\frac{1}{16}-1-\frac{1}{10^{6}}\right)\cdot \pi(x_0)>\frac{1}{25}\cdot\pi(x_0).\qedhere
\end{equation}
\end{proof}
\textit{Proof of \thmref{th:dis}:} Fix an $x\ge X_0$. Now choose $c=1+\eta_1$ (which is the smallest among $1+\eta_1$, $1.09$ and $1.02$) and $\delta=10^{-5}$ (which is the smallest among $10^{-5}$, $9\times 10^{-4}$ and $1/25$). Note here that both $c$ and $\delta$ are independent of $x$.

From propositions \ref{prop:proportion1}, \ref{prop:v_2} and \ref{prop:v_3}, for each large enough $x$, we get an $l_x\in\{1,2,3\}$ such that $V_{l_x}(c,\bullet;x)> \delta\cdot\pi(x)$. Also note that for any $p\in V_{l_x}(c,\bullet;x)$,
\begin{equation}
\max\{|\lambda_F(p^i)|: i=1,2,3\}\ge |\lambda_F(p^{l_x})|\ge c.
\end{equation}
Thus for any $x\ge X_0$, there exists an $l_x\in\{1,2,3\}$ such that 
\begin{equation}
\{p\le x : \max\{|\lambda_F(p^i)|: i=1,2,3\}\ge c\}\supseteq V_{l_x}(c,\bullet;x).
\end{equation}
This completes the proof of \thmref{th:dis} since both $c$ and $\delta$ are independent of $x$.
\begin{rmk}
Note here that the numerical values used in this section are not optimized. This is because it does not improve the Omega result that we are after.
\end{rmk}

\section{Proof of Theorem \ref{th:omega}}
By corollary \ref{cor:dis} of Theorem \ref{th:dis}, there exist constants $C>1$, $\delta>0$ and an integer $1\le r\le 3$ (depending on $N$) such that
\begin{equation}\label{fixr}
\#\{p\le N : |\lambda(p^r)|\ge C\}>\delta\cdot \pi(N),
\end{equation} 
for infinitely many integers $N$. Fix the integer $r$ from (\ref{fixr}) and denote the set on the LHS by $B_N$. We now use standard techniques (see \cite{murty1983oscillations} for similar arguments) to prove the $\Omega_{\pm}$ result. Let $B_N^{+}=\{p\le N : \lambda(p^r)\ge C\}\subset B_N$ and $B_N^{-}=\{p\le N : \lambda(p^r)\le -C\}\subset B_N$. Since $B_N>\delta \cdot \pi(N)$, either $B_N^{+}>\delta_1\cdot\pi(N)$ for some $\delta_1>0$ or $B_N^{-}>\delta_2\cdot\pi(N)$ for some $\delta_2>0$. 

If $B_N^{+}>\delta_1\cdot\pi(N)$, choose an integer $n$ as follows.
\begin{equation}\label{eq:n}
n=\prod_{p\in B_N^{+}}p^r.
\end{equation}
Note here that $r$ varies with $N$ and thus $n$. But this is not a cause of concern since $r\le 3$.

With this choice of $n$, we have
\begin{equation}
\lambda_F(n)=\prod_{p\in B_N^+}\lambda_F(p^r).
\end{equation}
Thus
\begin{equation*}
\lambda_F(n)\ge C^{|B_N^+|}> C^{\delta_1\cdot \pi(N)}\ge C^{\delta_1 c_1\cdot \frac{N}{\log N}}=\exp\left(c_0 \frac{N}{\log N}\right),
\end{equation*}
where we choose constants $c_1$ and $c_2$ such that $c_1\frac{N}{\log N}\le \pi(N)\le c_2\frac{N}{\log N}$ for all large $N$. Now from (\ref{eq:n}) we have
\begin{equation}\label{eq:nx}
\log n=r\sum_{p\in B_N^+}\log p\le r\sum_{p\le N}\log p\sim rN.
\end{equation}
Also note that
\begin{equation*}
\log n\ge r\log 2\cdot |B_N^+|\gg \frac{N}{\log N},
\end{equation*}
from which we get, $\log N\ll \log\log n$. Hence for some constant $c$,
\begin{equation}
\lambda_F(n)\gg\exp\left(\frac{c \log n}{\log \log n}\right).
\end{equation}
If $B_N^{-}>\delta_2\cdot\pi(N)$, we take $n$ to be product of even number of primes in $B_N^{-}$ and proceed as above.

Now to prove the $\Omega_{-}$ result consider the following. If $B_N^{+}>\delta_1\cdot\pi(N)$, then we proceed as follows. We know that there exists a $n_0\in \mbf Z$ such that $\lambda_F(n_0)<0$ (see \cite{kohnen2007sign}). Now let 
\begin{equation}\label{eq:n-}
n=n_0\underset{(p,n_0)=1}{\prod_{p\in B_N^+}}p^r.
\end{equation}
Thus $\lambda_F(n)=\lambda_F(n_0)\underset{(p,n_0)=1}{\prod_{p\in B_N^+}}\lambda_F(p^r)$. Now proceeding as above we get
\begin{equation}
-\lambda_F(n)\gg \exp\left(\frac{c \log n}{\log \log n}\right).
\end{equation}
If $B_N^{-}>\delta_2\cdot\pi(N)$, we take $n$ to be product of odd number of primes in $B_N^{-}$ and proceed as above. This completes the proof of Theorem \ref{th:omega}.

\end{document}